\documentclass[10pt]{article}
\usepackage{amsmath, amssymb, amsthm, amscd}
\numberwithin{equation}{section}

\def\p{\partial}

\def\cH{{\cal H}}

\def\cH{{\mathcal H}}

\newtheorem{prop}{Proposition}[section]
\newtheorem{theo}[prop]{Theorem}
\newtheorem{lemma}[prop]{Lemma}

\let\lra=\longrightarrow

\def\mapright\#1{\,\smash{\mathop{\lra}\limits^{\#1}}\,}

\begin{document}
\title{The Donaldson equation}
\author{Weiyong He \footnote{The author is partially supported by a PIMS postdoc fellowship.}}
\date{}
\maketitle

\section{Introduction}
S. Donaldson \cite{Donaldson2007} introduced a Weil-Peterson type
metric on the space of volume forms (normalized) on any Riemannian
manifold $(X, g)$ with fixed total volume. This infinite
dimensional space can be parameterized by all smooth functions
such that
\[
{\cal H} = \{\phi\in  C^\infty (X): 1 + \triangle_g \phi > 0\}.
\]
This is a locally Euclidean space.  The tangent space is exactly
$C^\infty(X)$ up to addition of some constants. The  metric is defined by
\[
\|\delta \phi\|_\phi^2 =  \int_X\; (\delta \phi)^2 (1 +
\triangle_g \phi) dg.
\]
The energy function on a path $\Phi: [0, 1]\rightarrow \cH$ is
defined as
\[
E(\Phi(t))=\int_0^1\int_X|\dot \Phi|^2(1+\triangle \Phi)dg.
\]
Then, the geodesic equation is
\begin{equation}\label{E-1-1}
\Phi_{tt}(1+\triangle \Phi) -|\nabla \Phi_t|^2_g = 0.
\end{equation}
This is a degenerated elliptic equation. To approach this
equation, Donaldson introduced a perturbed of the geodesic
equation
\begin{equation}\label{E-1-2}
\Phi_{tt}(1+\triangle \Phi) -|\nabla \Phi_t|^2_g =\epsilon,
\end{equation}
for any $\epsilon>0.$ The equation (\ref{E-1-2}) can be also
formulated  as the other two equivalent free boundary problems
according to \cite{Donaldson2007}. In joint work with X. Chen \cite{Chen-He},  we  get a smooth solution of the equation (\ref{E-1-2}) and a
weakly $C^2$ solution of the geodesic equation (\ref{E-1-1}),
where the a priori estimates on $|\Phi|_{C^1}$,  $\triangle \Phi$, $\Phi_{tt}$, $\nabla \Phi_t$  are independent of $\inf\epsilon$,
Using these solutions, we prove that $\cH$ is a
non-positively curved metric space, parallel to the result of the
space of K\"ahler metrics \cite{Chen00}.

From the PDE point of view, the equations (\ref{E-1-1}) and
(\ref{E-1-2}) are relevant to the operator
\begin{equation}\label{E-1-3}
Q(D^2\Phi)=\Phi_{tt}(1+\triangle \Phi) -|\nabla \Phi_t|^2.
\end{equation}
In this short note, we want to solve the following Dirichlet problem
\begin{equation}\label{E-1-4}
Q{(D^2\Phi)}=f,
\end{equation}
with boundary condition
\[\Phi(\cdot, 0)=\phi_0, \Phi(\cdot, 1)=\phi_1,\]
where $f\in C^{\infty}(X\times [0, 1])$ is a positive function and
$\phi_0, \phi_1\in \cH$. We call the equation (\ref{E-1-4})
Donladson equation and the operator $Q$ Donaldson operator. In the paper
\cite{Chen-He}, the fact that $f=\epsilon$ is a constant is used
crucially to get a priori estimates. We  notice that the equation can be still solved provided $f>0$, while $\inf f>0$ is actually used crucially to get a uniform $C^1$ bound below. We obtain
\begin{theo}\label{T-1-1}
Let $(X, g)$ be a compact Riemannian manifold and $f\in
C^{k}(X\times [0, 1])$ with $k\geq 2$ is a positive function. The
Dirichlet problem (\ref{E-1-4}) has a unique  solution $\Phi(x,
t)\in C^{k+1, \beta}(X\times [0, 1])$ for any  $\beta\in [0, 1)$.
Moreover, \[ 1+\triangle \Phi>0
\] for any $t\in [0, 1]$.
\end{theo}

\noindent{\bf Acknowledgement:} The author would like to thank Prof. X. Chen for constant support and encouragements.

\section{A priori estimates}
In this section we derive the {\it a priori} estimates for the
Donaldson equation
\begin{equation}\label{E-2-1}
Q{(D^2\Phi)}=f,
\end{equation}
with boundary condition
\[\Phi(\cdot, 0)=\phi_0, \Phi(\cdot, 1)=\phi_1,\]
where $f$ is a positive smooth function on $X\times [0, 1]$. The
linearized operator is given by
\[
dQ(h)=\Phi_{tt}\triangle h+(1+\triangle \Phi)h_{tt}-2\langle
\nabla h_t, \nabla\Phi_t\rangle.
\]
 Recall the concavity for the
Donaldson equation.
\begin{lemma}\label{L-2-1}(Donaldson \cite{Donaldson2007}) 1. If $A>0$, then
$Q(A)>0$ and if $A\geq 0$, $Q(A)\geq 0$.\\

2. If $A, B$ are two matrices with $Q(A)=Q(B)>0,$ and if the
entries $A_{00}, B_{00}$ are positive then for any $s\in [0, 1]$,
\[
Q(sA+(1-s)B)\geq Q(A), Q(A-B)\leq 0.
\]
Moreover, strict inequality holds if the corresponding arguments
are not the same.
\end{lemma}
We have its equivalent form.
\begin{lemma}\label{L-2-2} Consider the function
\[
f(x, y, z_1, \cdots, z_n)=\log{\left(xy-\sum z_i^2\right)}.
\]
Then $f$ is concave when $x>0, y>0, xy-\sum z_i^2>0$.
\end{lemma}
We will use the following notations. At any point $p\in X\times
[0, 1]$, take local coordinates $( x_1, \cdots, x_n, t)$. We can
always diagonalize the metric tensor $g$ as
$g_{ij}(p)=\delta_{ij}, \p_kg_{ij}(p)=0$. We will use, for any
smooth function $f$ on $X\times [0, 1]$, the following notations
\[
\triangle f_i=\triangle (f_i), ~~\triangle f_{ij}=\triangle
(f_{ij}), ~~\triangle f_{,i}=(\triangle f)_{, i}~~\mbox{and}~~
\triangle f,_{ij}=(\triangle f)_{ij}.
\]
For any function $f,$ $f_i, f_{ij}$ etc are covariant derivatives.
By Weitzenbock formula, we have
\begin{equation}\label{E-2-2}
\triangle f_i=\triangle f,_{i}+R_{ij}f_j,
\end{equation}
where $R_{ij}$ is the Ricci tensor of the metric $g$.
\subsection{$C^{0}$ estimates}
Denote \[\Phi_a=at(1-t)+(1-t)\phi_0+t\phi_1\] for any number $a$.
The $C^{0}$ estimate is similar as in \cite{Chen-He}. For the sake
of the completeness, we include the proof here.
\begin{lemma}\label{L-2-3}
If $\Phi$ satisfies (\ref{E-2-1}), then for some $a$ big enough,
\[
\Phi_{-a}\leq \Phi\leq (1-t)\phi_0+t\phi_1.
\]
\end{lemma}
\begin{proof}First we have
\[\Phi_{tt}>0.
\]
It follows that
\[
\frac{\Phi(\cdot, t)-\Phi(\cdot, 0)}{t-0}< \frac{\Phi(\cdot,
1)-\Phi(\cdot, t)}{1-t}.
\]
Namely
\[
\Phi(t)<(1-t)\phi_0+t\phi_1.
\]
Note $\Phi=\Phi_{-a}$ on the boundary. If $\Phi<\Phi_{-a}$ for
some point, then $\Phi-\Phi_{-a}$ obtains its minimum in the
interior, say at $p$. Then $D^{2}\Phi\geq D^{2}\Phi_{-a}$ at $p$.
Note
\[
Q(D^2\Phi)=f, ~~\mbox{and}~~ Q(D^2\Phi_{-a})=2a((1-t)\triangle
\phi_0+t\triangle \phi_1)-|\nabla\phi_1-\nabla\phi_0|^2.
\]
If $a$ is sufficiently big, we know that
\begin{equation}\label{E-2-3} Q(D^2\Phi)<Q(D^2\Phi_{-a}).
\end{equation}
Let $A$ be a $(n+2)\times(n+2)$ symmetric matrix such that the
$(n+1)\times(n+1)$ block of $A$ is $D^2\Phi_{-a}$, and
$A_{i(n+2)}=A_{(n+2)i}=0$ for $1\leq i\leq n+1$,
$A_{(n+2)(n+2)}=1$. Let $B$ be a $(n+2)\times(n+2)$ symmetric
matrix such that the $(n+1)\times(n+1)$ block of $B$ is $D^2\Phi$
and $B_{i(n+2)}=B_{(n+2)i}=0$ for $1\leq i\leq n+1$,
$B_{(n+2)(n+2)}=\lambda$. $\lambda$ is a constant satisfying
\[
Q(B)=\Phi_{tt}(\lambda+\triangle
\Phi)-\Phi_{tk}^2=Q(A)=Q(D^2\Phi_{-a}).
\]
We know that $\lambda>1$ by (\ref{E-2-3}). It follows  that
$Q(B-A)<0$. But $B-A$ is semi-positive definite, $Q(B-A)\geq0.$
Contradiction.
\end{proof}
\subsection{$C^1$ estimates} To get a $C^{1}$
estimate independent of $\epsilon$, in particular when
$\epsilon\rightarrow 0$,  the fact that $\epsilon$ is a
constant is used heavily in \cite{Chen-He}. In general, the required estimates can be obtained depending on  $\inf f>0$. 
\begin{lemma}Suppose that $\Phi$ satisfies (\ref{E-2-1}), then
there is a uniform constant $C$ depending on $\inf f>0, |f|_{C^1}$
and the boundary data, such that
\[
|\nabla \Phi|\leq C, |\Phi_t|\leq C.
\]
\end{lemma}
\begin{proof}Since $\Phi_{tt}>0$,  $\Phi_t$ obtains its
maximum on the boundary. By Lemma \ref{L-2-3}, it is easy to see
that $|\Phi_t|$ is bounded on the boundary. To bound $\nabla
\Phi$, take
\[
h=\frac{1}{2}\left(|\nabla\Phi|^2+b\Phi^2\right),
\]
where $b$ is a constant determined later. We want to show that $h$
is bounded. Namely, there exists a constant $C$ depending only on
$\inf f$, $|f|_{C^1}$ and the boundary data such that
\[ \max h\leq C.
\]
Since $h$ is uniformly bounded on the boundary,  we assume  $h$
takes its maximum at $(p, t_0)\in X\times (0, 1)$. Taking
derivative, we get that
\begin{eqnarray}\label{E-2-4}
h_t&=&\Phi_{tk}\Phi_k+b\Phi_t\Phi, ~~~
h_k=\Phi_{ik}\Phi_i+b\Phi_k\Phi,\nonumber\\
h_{tt}&=&\Phi_{ttk}\Phi_k+\Phi_{tk}^2+b(\Phi_{tt}\Phi+\Phi_t^2),\nonumber\\
h_{tk}&=&\Phi_{tik}\Phi_i+\Phi_{ti}\Phi_{ik}+b(\Phi_{tk}\Phi+\Phi_t\Phi_k),\nonumber\\
\triangle h&=&\Phi_{ikk}\Phi_i+\Phi^2_{ik}+b(\triangle
\Phi\Phi+\Phi_k^2),\nonumber\\
&=&\triangle \Phi_{, i}\Phi_i+\Phi^2_{ik}+b(\triangle
\Phi\Phi+\Phi_k^2)+R_{ij}\Phi_i\Phi_j,
\end{eqnarray}
where $R_{ij}$ is the Ricci curvature of $(X, g)$. It follows that
\begin{eqnarray}\label{E-2-5}
dQ(h)&=&\Phi_{tt}\triangle h+(1+\triangle
\Phi)h_{tt}-2\Phi_{tk}h_{tk}\nonumber\\
&=&\Phi_{tt}\left(\triangle \Phi_{,
i}\Phi_i+\Phi^2_{ik}+b(\triangle
\Phi\Phi+\Phi_k^2)\right)+\Phi_{tt}R_{ij}\Phi_i\Phi_j\nonumber\\
&&+(1+\triangle
\Phi)\left(\Phi_{ttk}\Phi_k+\Phi_{tk}^2+b(\Phi_{tt}\Phi+\Phi_t^2)\right)\nonumber\\
&&-2\Phi_{tk}\left(\Phi_{tik}\Phi_i+\Phi_{ti}\Phi_{ik}+b(\Phi_{tk}\Phi+\Phi_t\Phi_k)\right)\nonumber\\
&=&
\Phi_{tt}\Phi_{ik}^2+(1+\triangle\Phi)\Phi_{tk}^2-2\Phi_{ti}\Phi_{ik}\Phi_{tk}\nonumber\\
&&+b(\Phi_{tt}\Phi_k^2+(1+\triangle\Phi)\Phi_t^2-2\Phi_{tk}\Phi_t\Phi_k)\nonumber\\
&&+\Phi_{k}(\Phi_{tt}\triangle\Phi_{,k}+(1+\triangle
\Phi)\Phi_{ttk}-2\Phi_{ti}\Phi_{tik})\nonumber\\
&&+b\Phi(2\Phi_{tt}\triangle
\Phi+\Phi_{tt}-2\Phi_{tk}^2)+\Phi_{tt}R_{ij}\Phi_{i}\Phi_{j}.
\end{eqnarray}
Taking derivative of (\ref{E-2-1}), we can get that
\begin{eqnarray}\label{E-2-6}
\Phi_{ttk}(1+\triangle
\Phi)+\Phi_{tt}\triangle \Phi_{, k}-2\Phi_{itk}\Phi_{it}&=&f_k,\\
\label{E-2-7} \Phi_{ttt}(1+\triangle
\Phi)+\Phi_{tt}\triangle\Phi_{t}-2\Phi_{itt}\Phi_{it}&=&f_{t}.
\end{eqnarray}
By (\ref{E-2-5}) and (\ref{E-2-6}), we have
\begin{eqnarray}\label{E-2-8}
dQ(h)&=&\Phi_{tt}\Phi_{ik}^2+(1+\triangle\Phi)\Phi_{tk}^2-2\Phi_{ti}\Phi_{ik}\Phi_{tk}\nonumber\\
&&+b(\Phi_{tt}\Phi_k^2+(1+\triangle\Phi)\Phi_t^2-2\Phi_{tk}\Phi_t\Phi_k)\nonumber\\
&&+\Phi_{k}f_k-b\Phi_{tt}+\Phi_{tt}R_{ij}\Phi_i\Phi_j.\end{eqnarray}
Note at the point $(p, t_0)$, $h_t=h_k=0$, it follows that
\[
\Phi_{tk}\Phi_t=-b\Phi\Phi_t.
\]
We can get from (\ref{E-2-8}) that
\begin{eqnarray}\label{E-2-9}
dQ(h)&=&\Phi_{tt}\Phi_{ik}^2+(1+\triangle\Phi)\Phi_{tk}^2-2\Phi_{ti}\Phi_{ik}\Phi_{tk}\nonumber\\
&&+b(\Phi_{tt}\Phi_k^2+(1+\triangle\Phi)\Phi_t^2+2b\Phi\Phi_{t}^2)\nonumber\\
&&+\Phi_{k}f_k-b\Phi_{tt}+\Phi_{tt}R_{ij}\Phi_i\Phi_j\nonumber\\
&>&\Phi_{tt}\left(\frac{1}{2}b|\nabla\Phi|^2-b\Phi-C_0|\nabla\Phi|^2\right)\nonumber\\
&&+\frac{1}{2}b\Phi_{tt}|\nabla\Phi|^2+(1+\triangle\Phi)|\Phi_t|^2+\Phi_kf_k\nonumber\\
&>&\Phi_{tt}\left(\frac{1}{2}b|\nabla\Phi|^2-b\Phi-C_0|\nabla\Phi|^2\right)\nonumber\\
&&+(\sqrt{2bf}|\Phi_t|-|\nabla f|)|\nabla\Phi|,\end{eqnarray}
where $C_0=1+ \max|R_{ij}|$ is a constant. If $\Phi$ solves
(\ref{E-2-1}) with boundary condition
\[\Phi(\cdot, 0)=\phi_0, \Phi(\cdot, 1)=\phi_1,
\] then $\tilde \Phi=\Phi+At$ solves (\ref{E-2-1}) with  boundary
condition
\[
\tilde \Phi(x, 0)=\phi_0,~~\tilde\Phi(x, 1)=\phi_1+A,
\]
where $A$ is any constant. Since $|\Phi_{t}|$ and $|\Phi|$ are
bounded, we can choose normalization ($A$ big enough) such that
for any $(x, t)$,  $|\Phi_t|\geq 1.$ Choose $b$ such that \[b=\max
\left(\frac{|\nabla f|}{\sqrt {f}}, 4C_0\right).\] At the point
$(p, t_0)$, $dQ(h)\leq 0$, it follows from (\ref{E-2-9}) that
\[
|\nabla\Phi|^2(p)<\frac{b\Phi}{C_0}.
\]

\end{proof}
\subsection{$C^{2}$ estimates}
The $C^2$ estimates are only slight different with the case
$f=\epsilon$. First we have the following interior estimates.
\begin{lemma}
Suppose that $\Phi$ satisfies (\ref{E-2-1}), then there is a
uniform positive constants $C_1$ depending on $\inf f>0,
|f|_{C^1}, |f|_{C^2}$ and the boundary data, such that
\[
0<\Phi_{tt}+1+\triangle \Phi\leq C_1(1+\max_{\p (X\times [0,
1])}|\Phi_{tt}|).
\]
\end{lemma}
\begin{proof}It is clear that
\[\Phi_{tt}+1+\triangle \Phi>0.
\]
Take \[ F=\frac{1}{2}bt^2-b\Phi,~~h=\Phi_{tt}+1+\triangle \Phi,
~~\mbox{and}~~\tilde h=\exp{(F)}h,~~
\]
where $b$ is some constant determined later.  We want to show that
$\tilde h$ obtains its maximum on the boundary. If not , suppose
$h$ obtains its maximum at the point $(p, t_0)\in X\times (0, 1)$.
Taking derivative,
\[
\tilde h_t=\exp(F)(F_th+h_t), ~~~ \tilde h_k=\exp(F)(F_kh+h_k)
\] and
\[
\tilde h_{tt}=\exp(F)(h_{tt}+F_{tt}h+2F_th_t+hF_t^2), ~~ \tilde
h_{kk}=\exp(F)(h_{kk}+F_{kk}h+2F_kh_k+hF_k^2).
\]
Also we have
\[
\tilde h_{tk}=\exp(F)(h_{tk}+h_tF_k+hF_{tk}+hF_kF_t+h_kF_t).
\]
Note at the point $(p, t_0)$, $\tilde h_t=\tilde h_k=0$. It
follows that
\[
h_t+hF_t=0,~~~h_k+hF_k=0.
\]
We can calculate that at the point $(p, t_0)$
\begin{eqnarray}\label{E-2-10}dQ(\tilde h)&=&\Phi_{tt}\triangle \tilde h+(1+\triangle
\Phi)\tilde h_{tt}-2\Phi_{tk}\tilde h_{tk}\nonumber\\
&=&\Phi_{tt}\exp(F)(\triangle h+h\triangle F-hF_k^2)\nonumber\\
&&+(1+\triangle \Phi)\exp(F)(h_{tt}+hF_{tt}-hF_t^2)\nonumber\\
&&-2\Phi_{tk}\exp(F)(h_{tk}+hF_{tk}-hF_kF_t)\nonumber\\
&=&\exp(F)(dQ(h)+hdQ(F)-P(h, F)),
\end{eqnarray}
where
\[
P(h, F)=h(\Phi_{tt}F_k^2+(1+\triangle
\Phi)F_t^2-2\Phi_{tk}F_tF_k).
\]
Now we carry out $dQ(F), dQ(h)$. It is clear that
\[
dQ(F)=b(1+\triangle \Phi+\Phi_{tt}-2f).
\]
Taking derivative, we have
\begin{eqnarray*}
 h_t&=&\Phi_{ttt}+\triangle \Phi_t,~~~  h_{tt}=\Phi_{tttt}+\triangle \Phi_{tt}\nonumber\\
h_k&=&\Phi_{ttk}+\triangle \Phi,_{k},~~~ \triangle =\triangle
\Phi_{tt}+\triangle ^2\Phi,~~~ \tilde h_{tk}=\Phi_{tttk}+\triangle
\Phi_{, tk}.
\end{eqnarray*}
We calculate
\begin{eqnarray}\label{E-2-11}
dQ(h)&=&\Phi_{tt}\triangle  h+(1+\triangle \Phi)
h_{tt}-2\Phi_{tk} h_{tk}\nonumber\\
&=&(1+\triangle \Phi)(\Phi_{tttt}+\triangle
\Phi_{tt})+\Phi_{tt}(\triangle \Phi_{tt}+\triangle
^2\Phi)\nonumber\\
&&-2\Phi_{tk}(\Phi_{tttk}+\triangle \Phi_{, tk}).
\end{eqnarray}
Taking derivative of (\ref{E-2-6}) and (\ref{E-2-7}), we have
\begin{eqnarray}\label{E-2-12}
\Phi_{tt}\triangle \Phi_{tt}+(1+\triangle
\Phi)\Phi_{tttt}-2\Phi_{tk}\Phi_{tttk} +2\Phi_{ttt}\triangle
\Phi_t-2\Phi_{ttk}^2=f_{tt},&& \\\label{E-2-13}
\Phi_{tt}\triangle^2 \Phi+(1+\triangle \Phi)\triangle
\Phi_{tt}-2\Phi_{ti}\triangle \Phi_{ti}
+2\Phi_{ttk}\triangle\Phi_{,k}-2\Phi_{tik}^2=\triangle f.&&
\end{eqnarray}
It follows that
\begin{eqnarray}\label{E-2-14}
dQ(h)=2\Phi_{ttk}^2+2\Phi_{tik}^2 -2\Phi_{ttt}\triangle
\Phi_t-2\Phi_{ttk}\triangle
\Phi_{,k}+2R_{ij}\Phi_{ti}\Phi_{tj}+f_{tt}+\triangle f.\nonumber\\
\end{eqnarray}
Denote \begin{equation}\label{E-2-15}
L=\Phi_{ttk}^2-\Phi_{ttt}\triangle \Phi_t,
M=\Phi_{tij}^2-\Phi_{ttk}\triangle \Phi_{, k}.
\end{equation}
By (\ref{E-2-6}) and (\ref{E-2-7}), we get that
\[
\Phi_{tt}L=
\Phi_{tt}\Phi_{ttk}^2+(1+\triangle)\Phi_{ttt}^2-2\Phi_{tk}\Phi_{ttk}\Phi_{ttt},\]
and
\[
\Phi_{tt}M
=\Phi_{tt}\Phi_{tij}+(1+\triangle)\Phi_{ttk}^2-2\Phi_{tk}\Phi_{tik}\Phi_{tti}.\]
It follows that $L, M\geq 0.$ It is clear that
\[R_{ij}\Phi_{ti}\Phi_{tj}\geq -C_0(\Phi_{tt}+1+\triangle\Phi)^2,
\]
where $C_0=1+\max|R_{ij}|.$ It follows that
\[
dQ(h)\geq-C_0(\Phi_{tt}+1+\triangle \Phi)^2-|f|_{C^2}.
\]
It is also easy to get that
\[
P(h, f)\leq C_2(\Phi_{tt}+1+\triangle \Phi)^2,
\]
where $C_2$ is constant depending on $|\Phi|_{C^1}$. We can get
that
\[
dQ(\tilde h)>\exp(F)((b-C_0-C_2)(\Phi_{tt}+1+\triangle
\Phi)^2-2bf(\Phi_{tt}+1+\triangle \Phi)-|f|_{C^2}).
\]
Note at the point $(p, t_0)$, $dQ(\tilde h)\leq 0$. Take \[
b=C_0+C_2+1,
\]
we have at the point $(p, t_0)$
\[
\Phi_{tt}+1+\triangle \Phi\leq C_3(|f|_{C^2}, \inf f).
\]
Since $\exp(F)(\Phi_{tt}+1+\triangle \Phi)$ obtain its maximum at
$(p, t_0)$, it follows that
\[
\Phi_{tt}+1+\triangle \Phi\leq C_4.
\]
It means that either $\exp(F)(\Phi_{tt}+1+\triangle \Phi)$ obtains
its maximum on the boundary, or $\Phi_{tt}+1+\triangle \Phi$ is
uniformly bounded. In any case, we have
\[
0<\Phi_{tt}+1+\triangle \Phi\leq C_1(1+\max_{\p (X\times [0,
1])}|\Phi_{tt}|).
\]
\end{proof}
The boundary $C^2$ estimates follow exactly the same as in
\cite{Chen-He}. \begin{lemma} If $\Phi$ is a solution of
(\ref{E-2-1}), then $\Phi$ satisfies the following {\it a priori}
estimate
\[
|\triangle \Phi|\leq C~~, |\Phi_{tk}|\leq C,~~ |\Phi_{tt}|\leq C,
\]
where $C$ is a universal constant depending on $\inf f, |f|_{C^2}$
and the boundary data.
\end{lemma}
The H\"older estimate of $D^2\Phi$ follows from Evans-Krylov theory using  the concavity
of $\log Q$.  Once we get the H\"older estimates of $D^2\Phi$, the standard
boot-strapping argument gives all higher order derivatives of
$\Phi$. 
\section{Solve the equation}
To solve the Donaldson equation for general $f$, we consider the
following continuity family for $s\in [0, 1]$
\begin{equation}\label{E-3-1}
Q(D^2\Phi)=(1-s)Q(D^2\Phi_{-a})+sf,
\end{equation}
with the boundary condition
\[
\Phi(\cdot, 0, s)=\phi_0, \Phi(\cdot, 1, s)=\phi_1,
\]
where $\Phi_{-a}=-at(1-t)+(1-t)\phi_0+t\phi_1$. When $a$ is big
enough, $Q(D^2\Phi_{-a})$ is positive and bounded away from $0$.
We shall now prove that if $f\in C^{k}(X\times [0, 1])$ with
$k\geq 2$ then we can find of solution of (\ref{E-2-1}) such that
$\Phi\in C^{k+1, \beta}(X\times [0, 1])$ for any $0\leq \beta<1$.
Consider the set
\[S=\left\{s\in[0, 1]:~\mbox{the equation (\ref{E-3-1}) has a solution in}~~C^{k-1, \beta}(X\times [0, 1])\right\}
\]
Obviously $0\in S$. Hence we need only show that $S$ is both open
and close. It is clear that $Q: C^{k+1, \beta}\rightarrow C^{k-1,
\beta}$ is open if \[ 1+\triangle\Phi>0~~ \mbox{and}~~
Q(D^2\Phi)>0.
\]
In this case $dQ$ is an invertible elliptic operator and openness
follows. The closeness of $S$ follows from the a prior estimates
derived in Section 2. Hence Theorem \ref{T-1-1} holds.
\bibliographystyle{plain}

\noindent Weiyong He\\
whe@math.ubc.ca\\
Department of Mathematics\\
University of British Columbia
\end{document}